\documentclass[12pt,a4paper]{article}
\usepackage{amsthm,amsmath,amssymb,amsfonts}
\newtheorem{thm}{Theorem}[section]
\newtheorem{cor}[thm]{Corollary}
\newtheorem{lem}[thm]{Lemma}
\theoremstyle{definition}

\theoremstyle{remark}
\newtheorem{rem}[thm]{Remark}
\numberwithin{equation}{section}
\begin{document}
\begin{center}
{\bf Complete systems of partial derivatives of entire functions and frequently hypercyclic operators\footnote{This research is partially supported by the Russian Foundation for Basic Research (projects 14-01-00720 and 14-01-97037)}}
\vskip 0.1cm
{\bf Vitaly E. Kim} \\ {\small Russian Academy of Sciences, Institute of mathematics with computing center, 
\\ 112 Chernyshevsky str., Ufa, Russia, 450008
\\ e-mail: kim@matem.anrb.ru}
\end{center}
\begin{abstract}
We find some sufficient conditions for a system of partial derivatives of an entire function to be complete in the space $H(\mathbb{C}^d)$ of all entire functions of $d$ variables. As an application of this result we describe new classes of frequently hypercyclic operators on $H(\mathbb{C}^d)$.
\end{abstract}
\section{Introduction}\label{sec1}
In this paper we consider the space $H(\mathbb{C}^d)$ of all entire functions on $\mathbb{C}^d$ with the topology of uniform convergence on compact subsets. As usual, $\mathbb{N}_{0}$ is the set of non-negative integers. For $z=(z_1,\cdots,z_d),w=(w_1,\cdots,w_d)\in\mathbb{C}^d$ and $n=(n_1,\cdots,n_d)\in\mathbb{N}_{0}^d$ we shall use the following notation: $z^n=z_1^{n_1}z_2^{n_2}\cdots z_d^{n_d}$, $n!=n_1!n_2!\cdots n_d!$, $\|n\|=n_1+\cdots +n_d$, $D^n=\frac{\partial^{\|n\|}}{\partial z_1^{n_1}\cdots\partial z_d^{n_d}}$, $\left\langle w,z\right\rangle=w_1z_1+\dots+w_dz_d$, $|z|=\max\{|z_j|,\;1\leq j\leq d\}$.
\par
Recall that a continuous linear operator $T$ on a topological vector space $X$ is called hypercyclic if there exists an element $x\in X$ such that its orbit $\mathrm{Orb}(T,x)=\{T^nx,\;n=0,1,\cdots\}$ is dense in $X$. Recently, F. Bayart and S. Grivaux \cite{21} have introduced a stronger notion of hypercyclicity, namely the concept of frequently hypercyclic operators. Recall that the lower density of the set $A\subset\mathbb{N}_{0}$ is defined as
$$\underline{\mathrm{dens}}(A)=\liminf_{N\rightarrow\infty}\frac{\#\{n\in A:n\leq N\}}{N},$$
where $\#$ denotes the cardinality of a set.
\par\noindent
{\bf Definition.} A continuous linear operator $T$ on a topological vector space $X$ is called frequently hypercyclic if there exists an element $x\in X$ such that, for any non-empty open subset $U$ of $X$,
$$\underline{\mathrm{dens}}\{n\in\mathbb{N}_{0}:T^nx\in U\}>0.$$
\par
We refer to \cite{13}, \cite{20} for additional information on hypercyclic and frequently hypercyclic operators.
\par
In the present paper we prove the following result.
\begin{thm}\label{th1}
Let $T_j:H(\mathbb{C}^d)\rightarrow H(\mathbb{C}^d)$, $j=1,\cdots,d$ be continuous linear operators such that
\begin{enumerate}
		\item $\bigcap_{j=1}^d\mathrm{ker}\,T_j\neq\{0\}$;
		\item $T_j$ satisfy the commutation relations:
\begin{equation}\label{f1}
	[T_j,\frac{\mathrm{\partial}}{\mathrm{\partial}z_k}]=\delta_{jk}a_jI;\,j,k=1,\cdots,d,
\end{equation}
where $\delta_{jk}$ is Kronecker's delta, $a_j\in\mathbb{C}\setminus\{0\}$ are some constants and $I$ is identity operator.
\end{enumerate}
Then each of the operators $T_j$, $j=1,\cdots,d$ is frequently hypercyclic.
\end{thm}
\begin{rem}
In Theorem \ref{th1} we consider constants $a_j\neq0$. Note that if $a_j=0$ in (\ref{f1}) for some $j$ then the operators $T_j$ commute with each partial differentiation operator. Recall that a continuous linear operator on $H(\mathbb{C}^d)$ that commutes with each partial differentiation operator is called a convolution operator. A theorem of Godefroy and Shapiro \cite{3} states that every convolution operator on $H(\mathbb{C}^d)$ that is not a scalar multiple of the identity is hypercyclic. In \cite{22} it was shown that these operators are even frequently hypercyclic.
\end{rem}
\begin{rem}
Theorem \ref{th1} is an extension of the author's result in \cite{23} that establishes hypercyclicity of operators satisfying the conditions of Theorem \ref{th1} for $d=1$.
\end{rem}
The paper is organized as follows. In Section \ref{sec2} we obtain results on completeness of systems translates and systems of partial derivatives of entire functions. In Section \ref{sec3} we prove Theorem \ref{th1} by application of the results of Section \ref{sec2}. In Section \ref{sec4} we provide some examples of operators that are frequently hypercyclic by Theorem \ref{th1}.
\section{Complete systems of entire functions}\label{sec2}
In this section we study completeness of some systems of entire functions in the space $H(\mathbb{C}^d)$. More precisely, we shall consider the systems of translates and the systems of all partial derivatives of an entire function. Recall that the system of elements of a topological vector space $X$ is said to be complete in $X$ if the linear span of the system is dense in $X$. For $\lambda\in\mathbb{C}^d$ and $f\in H(\mathbb{C}^d)$ let us denote $S_\lambda f(z)\equiv f(z+\lambda)$.
\par
In this section we shall need some facts from the theory of convolution operators on $H(\mathbb{C}^d)$. These facts can be found, for example, in \cite{24}. Recall that every convolution operator on $H(\mathbb{C}^d)$ is defined by some linear continuous functional on $H(\mathbb{C}^d)$. Denote by $H'(\mathbb{C}^d)$ the strong dual space of $H(\mathbb{C}^d)$. Let $F\in H'(\mathbb{C}^d)$. Then the corresponding convolution operator $M_F$ has the form:
$$M_F[f](\lambda)=(F,S_\lambda f),\;\lambda\in\mathbb{C}^d,\;f\in H(\mathbb{C}^d).$$
The convolution operator $M_F$ is a linear continuous operator from $H(\mathbb{C}^d)$ to $H(\mathbb{C}^d)$. The Laplace transform
\begin{equation}\label{eq99}
L:F\rightarrow\widehat{F}(\lambda)=(F,\exp(z\lambda)),	
\end{equation}
where $n=(n_1,\cdots,n_d)$, establishes a one-to-one correspondence (and moreover, topological isomorphism) between $H'(\mathbb{C}^d)$ and the space $P_{\mathbb{C}^d}$ of entire functions of exponential type on $\mathbb{C}^d$. The function $\widehat{F}$ is said to be a characteristic function for the operator $M_F$. Consider the Taylor expansion:
$$\widehat{F}(\lambda)=\sum_{\|n\|=0}^\infty\frac{b_n}{n!}\lambda^n.$$
Then $(F,f)=\sum_{\|n\|=0}^\infty a_nb_n$, where $f(z)=\sum_{\|n\|=0}^\infty a_nz^n$. Since
$$f(\xi+\lambda)=\sum_{\|n\|=0}^\infty\frac{D^nf(\lambda)}{n!}\xi^n,$$
then $M_F$ can be represented in the form:
\begin{equation}\label{eq100}
M_F[f](\lambda)=\sum_{\|n\|=0}^\infty\frac{b_nD^nf(\lambda)}{n!}.
\end{equation}
On the other hand, since
$$f(\xi+\lambda)=\sum_{\|n\|=0}^\infty\frac{D^nf(\xi)}{n!}\lambda^n,$$
then
\begin{equation}\label{eq300}
M_F[f](\lambda)=\sum_{\|n\|=0}^\infty\frac{(F,D^nf(\xi))}{n!}\lambda^n.	
\end{equation}
In the following theorem we establish necessary and sufficient conditions for systems of partial derivatives and systems of translates to be complete in $H(\mathbb{C}^d)$.
\begin{thm}\label{th100}
Let $f\in H(\mathbb{C}^d)$, $f\not\equiv0$. Then the following conditions are equivalent:
  \begin{enumerate}
		\item $f\not\in\mathrm{Ker}\,M_F$ for any $F\in H'(\mathbb{C}^d)\setminus\{0\}$;
		\item The system of translates $\{S_\lambda f,\;\lambda\in\Lambda\}$ is complete in $H(\mathbb{C}^d)$ for each open set $\Lambda\subset\mathbb{C}^d$;
		\item The system of partial derivatives $\{D^n f,n\in\mathbb{N}_{0}^d\}$ is complete in $H(\mathbb{C}^d)$.
	\end{enumerate}
\end{thm}
\begin{proof}
$1\Rightarrow 2$. Assume that the system $\{S_\lambda f,\;\lambda\in\Lambda\}$ is not complete in $H(\mathbb{C}^d)$ for some open set $\Lambda\subset\mathbb{C}^d$. Then according to the Approximation Principle (see e.g. \cite[Ch.2]{27}) there is a non-zero continuous linear functional $F$ on $H(\mathbb{C}^d)$ such that $(F, S_\lambda f)=0,\,\forall\lambda\in\Lambda$. Since $\Lambda$ is an open set, then $M_F[f](\lambda)=(F, S_\lambda f)=0,\,\forall\lambda\in\mathbb{C}^d$. Hence, $f\in\mathrm{Ker}\,M_F$.
\par
$2\Rightarrow 3$. Assume that the system $\{D^n f,n\in\mathbb{N}_{0}^d\}$ is not complete in $H(\mathbb{C}^d)$. Then there is $0\neq F\in H'(\mathbb{C}^d)$ such that $(F,D^n f)=0$, $\forall n\in\mathbb{N}_{0}^d$. Let $\Lambda\subset\mathbb{C}^d$ be some open set. Then 
$$(F,S_\lambda f)=M_F[f](\lambda)=\sum_{\|n\|=0}^\infty\frac{(F,D^nf)}{n!}\lambda^n=0,\;\forall\lambda\in\Lambda.$$
Hence, the system $\{S_\lambda f,\;\lambda\in\Lambda\}$ is not complete in $H(\mathbb{C}^d)$.
\par
$3\Rightarrow 1$. Assume that $f\in\mathrm{Ker}\,M_F$ for some $F\in H'(\mathbb{C}^d)\setminus\{0\}$. Then by (\ref{eq300})
$$M_F[f](\lambda)=\sum_{\|n\|=0}^\infty\frac{(F,D^nf)}{n!}\lambda^n=0,\,\forall\lambda\in\mathbb{C}^d.$$
Hence, $(F,D^nf)=0$, $\forall n\in\mathbb{N}_{0}^d$. Thus, the system $\{D^n f,n\in\mathbb{N}_{0}^d\}$ is not complete in $H(\mathbb{C}^d)$.	
\end{proof}
We need the following:
\begin{lem}\label{lem100}
Let the operators $T_j$, $j=1,\cdots,d$ satisfy the conditions of the Theorem \ref{th1}. Let $f\in \bigcap_{j=1}^d\mathrm{ker}\,T_j$, $f\not\equiv0$. Then, for all $k=(k_1,\dots,k_d),n=(n_1,\dots,n_d)\in\mathbb{N}_{0}^d$, 
\begin{equation}\label{eq102}
T^kD^nf=\begin{cases}
a^k\frac{n!}{(n-k!)}D^{n-k}f, & \text{if $k_j\leq n_j,\,j=1,\dots,d$},\\
0, & \text{otherwise},
\end{cases}
\end{equation}
where $T^k=T_1^{k_1}T_2^{k_2}\dots T_d^{k_d}$, $a^k=a_1^{k_1}a_2^{k_2}\cdots a_d^{k_d}$.
\end{lem}
\begin{proof}
Let us fix some arbitrary $j\in\mathbb{N}:1\leq j\leq d$ and $n\in\mathbb{N}_{0}^d$. From (\ref{f1}) it follows (see e.g. \cite[\S 16]{26}) that
\begin{equation}\label{eq114}
\Bigl[T_j,\frac{\mathrm{\partial}^{n_j}}{\mathrm{\partial}z_j^{n_j}}\Bigr]=a_jn_j\frac{\mathrm{\partial}^{n_j-1}}{\mathrm{\partial}z_j^{n_j-1}}.
\end{equation}
Let us put 
$$\check{f}(z)=\frac{\partial^{\|n\|-n_j}f(z)}{\partial z_1^{n_1}\cdots\partial z_{j-1}^{n_{j-1}}\partial z_{j+1}^{n_{j+1}}\cdots\partial z_d^{n_d}}.$$
Then $D^nf(z)\equiv\frac{\mathrm{\partial}^{n_j}}{\mathrm{\partial}z_j^{n_j}}\check{f}(z)$. Since $f\in\mathrm{Ker}\,T_j$, then it follows from (\ref{f1}) that $\check{f}\in\mathrm{Ker}\,T_j$. Hence, (\ref{eq114}) implies that 
\begin{equation}\label{eq115}
T_j^{k_j}\frac{\mathrm{\partial}^{n_j}}{\mathrm{\partial}z_j^{n_j}}\check{f}=a_j^{k_j}\frac{n_j!}{(n_j-k_j)!}\frac{\mathrm{\partial}^{n_j-k_j}}{\mathrm{\partial}z_j^{n_j-k_j}}\check{f},
\end{equation}
if $k_j\leq n_j$, and
\begin{equation}\label{eq116}
T_j^{k_j}\frac{\mathrm{\partial}^{n_j}}{\mathrm{\partial}z_j^{n_j}}\check{f}=0,
\end{equation}
if $k_j>n_j$. Since (\ref{eq115}), (\ref{eq116}) hold for an arbitrary $j\in\mathbb{N}:1\leq j\leq d$, the lemma then follows.
\end{proof}
Now we are ready to state and prove the main result of this section.
\begin{thm}\label{th2}
Let the operators $T_j$, $j=1,\cdots,d$ satisfy the conditions of the Theorem \ref{th1}. Let $f\in \bigcap_{j=1}^d\mathrm{ker}\,T_j$, $f\not\equiv0$. Then the system $\{D^n f,n\in\mathbb{N}_{0}^d\}$ is complete in $H(\mathbb{C}^d)$.
\end{thm}
\begin{proof}
Assume that there exists $F\in H'(\mathbb{C}^d)\setminus\{0\}$ such that $f\in\mathrm{Ker}\,M_F$. Then $T^kM_F[f]=0$, $\forall k\in\mathbb{N}_{0}^d$. Then by (\ref{eq99}), (\ref{eq100}) we have
\begin{equation}\label{eq101}
T^kM_F[f](z)=\sum_{\|n\|=0}^\infty\frac{D^n\widehat{F}(0)}{n!}T^kD^nf(z)\equiv0,\;\forall k\in\mathbb{N}_{0}^d.
\end{equation}
From (\ref{eq102}), (\ref{eq101}) we obtain
\begin{equation*}
\sum_{n_1=k_1,\dots,n_d=k_d}^\infty\frac{D^n\widehat{F}(0)}{(n-k)!}D^{n-k}f(z)\equiv0,\;\forall k\in\mathbb{N}_{0}^d,
\end{equation*}
or
\begin{equation}\label{eq104}
\sum_{\|n\|=0}^\infty\frac{D^{n+k}\widehat{F}(0)}{n!}D^nf(z)\equiv0,\;\forall k\in\mathbb{N}_{0}^d.
\end{equation}
Next, using the Taylor series expansion of $D^nf(z)$ we obtain
\begin{equation}\label{eq105}
\sum_{\|n\|=0}^\infty\frac{D^{n+k}\widehat{F}(0)}{n!}\sum_{\|m\|=0}^\infty\frac{D^{n+m}f(0)}{m!}z^m\equiv0,\;\forall k\in\mathbb{N}_{0}^d.
\end{equation}
Since the series in (\ref{eq105}) converges absolutely, then we can interchange the order of summation:
\begin{equation}\label{eq106}
\sum_{\|m\|=0}^\infty\frac{z^m}{m!}\sum_{\|n\|=0}^\infty\frac{D^{n+m}f(0)D^{n+k}\widehat{F}(0)}{n!}\equiv0,\;\forall k\in\mathbb{N}_{0}^d.
\end{equation}
From (\ref{eq106}) it follows that
\begin{equation*}
\sum_{\|n\|=0}^\infty\frac{D^{n+m}f(0)D^{n+k}\widehat{F}(0)}{n!}=0,\;\forall m,k\in\mathbb{N}_{0}^d.
\end{equation*}
Hence,
\begin{equation*}
\frac{\lambda^k}{k!}\sum_{\|n\|=0}^\infty\frac{D^{n+m}f(0)D^{n+k}\widehat{F}(0)}{n!}=0,\;\forall m,k\in\mathbb{N}_{0}^d,\,\forall\lambda\in\mathbb{C}^d,
\end{equation*}
and
\begin{equation}\label{eq109}
\sum_{\|k\|=0}^\infty\frac{\lambda^k}{k!}\sum_{\|n\|=0}^\infty\frac{D^{n+m}f(0)D^{n+k}\widehat{F}(0)}{n!}=0,\;\forall m\in\mathbb{N}_{0}^d,\,\forall\lambda\in\mathbb{C}^d.
\end{equation}
Since the series in (\ref{eq109}) converges absolutely, then we can interchange the order of summation:
\begin{equation}\label{eq110}
\sum_{\|n\|=0}^\infty\frac{D^{n+m}f(0)}{n!}\sum_{\|k\|=0}^\infty\frac{D^{n+k}\widehat{F}(0)}{k!}\lambda^k=0,\;\forall m\in\mathbb{N}_{0}^d,\,\forall\lambda\in\mathbb{C}^d.
\end{equation}
Finally, we can rewrite (\ref{eq110}) in the following form:
\begin{equation}\label{eq111}
\sum_{\|n\|=0}^\infty\frac{D^{n+m}f(0)}{n!}D^n\widehat{F}(\lambda)\equiv0,\;\forall m\in\mathbb{N}_{0}^d.
\end{equation}
Consider the infinite system of infinite-order homogeneous linear partial differential equations with constant coefficients
\begin{equation}\label{eq112}
\sum_{\|n\|=0}^\infty\frac{D^{n+m}f(0)}{n!}D^ng(\lambda)=0,\;g\in P_{\mathbb{C}^d},\;m\in\mathbb{N}_{0}^d.
\end{equation}
with characteristic functions
$$D^mf(z)=\sum_{\|n\|=0}^\infty\frac{D^{n+m}f(0)}{n!}z^n,\;m\in\mathbb{N}_{0}^d.$$
Note that the series in (\ref{eq112}) converges absolutely and uniformly for any function $g\in P_{\mathbb{C}^d}$. For each of the equations in the system (\ref{eq112}) denote by $W_m$ the space of solutions, $m\in\mathbb{N}_{0}^d$. Denote by $W$ the space of solutions of the system (\ref{eq112}). Then, obviously, $W=\bigcap W_m$. It is easy to see that $W$ is a translation-invariant subspace of $P_{\mathbb{C}^d}$ (i.e. if $g\in W$, then $S_ug\in W$, $\forall u\in\mathbb{C}^d$). 
\par
In \cite{25} it was shown that every translation-invariant subspace of $P_{\mathbb{C}^d}$ admits spectral synthesis. In particular, this means that  a translation-invariant subspace $V\subset P_{\mathbb{C}^d}$ is not trivial (i.e. $V\neq\{0\}$) if and only if it contains some exponential element of the form 
\begin{equation}\label{eq113}
\varphi_\mu(z)=P(z)\exp\left\langle \mu,z\right\rangle,\;\mu\in\mathbb{C}^d,
\end{equation}
where $P(z)\in H(\mathbb{C}^d)$ is some polynomial. From (\ref{eq111}) it follows that the function $\widehat{F}\in P_{\mathbb{C}^d}$ satisfy the system (\ref{eq112}). Since $\widehat{F}\in W$, then our initial assumption that $\widehat{F}\not\equiv0$ implies that $W\neq\{0\}$, and hence $W$ contains a function $\varphi_\mu$ of the form (\ref{eq113}). Therefore, $\varphi_\mu\in W_m$, $\forall m\in\mathbb{N}_{0}^d$. 
\par
Let us fix some arbitrary $r\in\mathbb{N}_{0}^d$. We are going to show that $\exp\left\langle \mu,z\right\rangle\in W_r$. Obviously, there exists some $q\in\mathbb{N}_{0}^d$ such that $D^qP(z)\equiv\mathrm{const}$. Note that $W_r$ is invariant under each partial differentiation operator. Then, for example, $\frac{\partial}{\partial z_1}\varphi_\mu\in W_r$. Hence, $\frac{\partial}{\partial z_1}\varphi_\mu-\mu_1\varphi_\mu=\frac{\partial\,P}{\partial z_1}\exp\left\langle \mu,z\right\rangle\in W_r$. By continuing this process we can show that $D^q(P)\exp\left\langle \mu,z\right\rangle\in W_r$ or, equivalently, $\exp\left\langle \mu,z\right\rangle\in W_r$.
\par
Then
$$\sum_{\|n\|=0}^\infty\frac{D^{n+r}f(0)}{n!}D^n(\exp\left\langle \mu,z\right\rangle)=0,$$
or, equivalently,
$$\sum_{\|n\|=0}^\infty\frac{D^{n+r}f(0)}{n!}\mu^n=D^rf(\mu)=0.$$
\par
Hence, $D^mf(\mu)=0,\;\forall m\in\mathbb{N}_{0}^d$. But this is impossible since $f\not\equiv0$. Hence, $W=\{0\}$. But this contradicts our initial assumption that $\widehat{F}\not\equiv0$. Hence, $f\not\in\mathrm{Ker}\,M_\Phi$ for any $\Phi\in H'(\mathbb{C}^d)\setminus\{0\}$. The theorem then follows from Theorem \ref{th100}.
\end{proof}
\begin{cor}
Let operators $T_j$, $j=1,\cdots,d$ satisfy the conditions of the Theorem \ref{th1}. Let $f\in \bigcap_{j=1}^d\mathrm{ker}\,T_j$, $f\not\equiv0$. Then the system of translates $\{S_\lambda f,\;\lambda\in\Lambda\}$ is complete in $H(\mathbb{C}^d)$ for each open set $\Lambda\subset\mathbb{C}^d$.
\end{cor}
\begin{rem}
The analogue of Theorem \ref{th2} for $d=1$ was previously proved in author's paper \cite{23} with the help of the result of L. Schwartz \cite{14} which says that every translation-invariant subspace of $H(\mathbb{C})$ admits spectral synthesis. The situation is more difficult if $d>1$ because there is a counter example on spectral synthesis in $H(\mathbb{C}^d)$, $d>1$, by D. Gurevich \cite{28}. In particular Gurevich proved that there exists translation-invariant subspace $U\neq\{0\}$ of $H(\mathbb{C}^d)$ that does not contain any exponential element. This is why we had to transform a system (\ref{eq104}) of equations in $H(\mathbb{C}^d)$ to a system (\ref{eq111}) in $P_{\mathbb{C}^d}$.
\end{rem}
\begin{rem}
In \cite{23} it was shown that the system $\{f^{(n)},\,n=0,1,\dots\}$ is complete in $H(\mathbb{C})$ if $f\in\mathrm{ker}\,T\setminus\{0\}$, where $T$ is a continuous linear operator on $H(\mathbb{C})$ such that $[T,\frac{\mathrm{d}}{\mathrm{d}\,z}]=aI$ for some $a\in\mathbb{C}\setminus\{0\}$. But in Theorem \ref{th2} the condition $f\in\mathrm{ker}\,T_j\setminus\{0\}$ would not be sufficient for completeness of the system $\{D^n f,n\in\mathbb{N}_{0}^d\}$. For example, consider the operator $T_1g(z)=\frac{\mathrm{\partial}g(z)}{\mathrm{\partial}z_1}-z_1g(z)$ on $H(\mathbb{C}^2)$. Then $T_1$ satisfies (\ref{f1}) and $f(z)=\exp(z_1^2/2)\in\mathrm{ker}\,T_1$. But $\frac{\mathrm{\partial}}{\mathrm{\partial}z_2}\exp(z_1^2/2)=0$ and by Theorem \ref{th100} the system $\{D^n f,n\in\mathbb{N}_{0}^2\}$ is not complete in $H(\mathbb{C}^2)$. This is why we need a stronger condition $f\in\bigcap_{j=1}^d\mathrm{ker}\,T_j$.
\end{rem}
\section{The proof of the main result}\label{sec3}
In this section we give a proof of Theorem \ref{th1}.
\begin{proof}
Let us fix some arbitrary $j:1\leq j\leq d$. We are going to show that the operator $T_j$ is frequently hypercyclic. The proof will follow by an application of the frequent hypercyclicity criterion \cite{21} (see also e.g. \cite[Ch. 9]{20}): it is enough to show that there is a dense subset $H_0$ of $H(\mathbb{C}^d)$ and a map $S_j:H_0\rightarrow H_0$ such that, for any $x\in H_0$,
\begin{description}
	\item[(i)] $\sum_{k=0}^\infty T_j^kx$ converges unconditionally;
	\item[(ii)] $\sum_{k=0}^\infty S_j^kx$ converges unconditionally;
	\item[(iii)] $T_jS_jx=x$.
\end{description}
Consider some function $f\in \bigcap_{j=1}^d\mathrm{ker}\,T_j$, $f\not\equiv0$. Let us put
$$H_0=\mathrm{span}\{D^n f,n\in\mathbb{N}_{0}^d\}.$$
Then by Theorem \ref{th2}, $H_0$ is a dense subset of $H(\mathbb{C}^d)$. From (\ref{eq116}) it follows that for any $x\in H_0$ there is $K\in\mathbb{N}$ such that $T_j^kx=0$, if $k\geq K$. Thus, condition (i) is fulfilled.
\par
Let us define a map $S_j$ in the following way:
\begin{equation*}
S_j[D^nf]=\frac{1}{a_j(n_j+1)}\frac{\mathrm{\partial}}{\mathrm{\partial}z_j}D^nf,\;n\in\mathbb{N}_{0}^d.
\end{equation*}
Therefore,
\begin{equation*}
S_j^k[D^nf]=\frac{n_j!}{a_j^k(n_j+k)!}\frac{\mathrm{\partial}^k}{\mathrm{\partial}z_j^k}D^nf,\;n\in\mathbb{N}_{0}^d,\;k\in\mathbb{N}_0.
\end{equation*}
Let us fix some $\varepsilon\in\mathbb{R}$ such that
\begin{equation}\label{eq202}
\varepsilon>\max_{1\leq s\leq d}\frac{1}{|a_s|}.	
\end{equation}
The topology on $H(\mathbb{C}^d)$ can be defined by a family of semi-norms 
$$p_m(g)=\sup_{z\in\Omega_m}|g(z)|,\;m\in\mathbb{N},$$
where $\Omega_m=\{z\in\mathbb{C}^d:|z|\leq m\varepsilon\}$. In order to prove (ii), it is enough to show that the series
\begin{equation}\label{eq201}
\sum_{k=0}^\infty p_m(S_j^kD^nf)
\end{equation}
converges for any $m\in\mathbb{N}$, $n\in\mathbb{N}_{0}^d$. 
\par
Using the Cauchy estimate for derivatives, we obtain for all $\;n\in\mathbb{N}_{0}^d,\;k\in\mathbb{N}_0,\;m\in\mathbb{N}$:
$$p_m(S_j^kD^nf)=\frac{n_j!}{a_j^k(n_j+k)!}p_m\Bigl(\frac{\mathrm{\partial}^k}{\mathrm{\partial}z_j^k}D^nf\Bigr)\leq\frac{n_j!k!p_{2m}(D^nf)}{|a_j|^k(n_j+k)!(m\varepsilon)^k}.$$
Then by (\ref{eq202}) we have
$$\lim_{k\rightarrow\infty}\sqrt[k]{p_m(S_j^kD^nf)}\leq\frac{1}{|a_j|m\varepsilon}<1,\;m\in\mathbb{N},\,n\in\mathbb{N}_{0}^d.$$
Therefore, the series (\ref{eq201}) converges for any $m\in\mathbb{N}$, $n\in\mathbb{N}_{0}^d$. This means that condition (ii) is fulfilled. Furthermore, $T_jS_jx=x$ for any $x\in H_0$ and thus, condition (iii) is fulfilled.
\end{proof}
\section{Examples}\label{sec4}
In this section we provide some examples of operators that are frequently hypercyclic by Theorem \ref{th1}.
\par
1) Let 
$$T_1g(z)=\frac{\mathrm{\partial}g(z)}{\mathrm{\partial}z_1}-z_1g(z),\;T_2g(z)=\frac{\mathrm{\partial}g(z)}{\mathrm{\partial}z_2}-z_2g(z),\;g\in H(\mathbb{C}^2).$$
Then $\bigcap_{j=1}^2\mathrm{ker}\,T_j$ contain, for instance, the function $\exp(z_1^2/2)\cdot\exp(z_2^2/2)$. Besides,
\begin{equation}\label{eq200}
[T_j,\frac{\mathrm{\partial}}{\mathrm{\partial}z_k}]=\delta_{jk}I;\,j,k=1,2.	
\end{equation}
Therefore, $T_1$ and $T_2$ are frequently hypercyclic by Theorem \ref{th1}.
\par
2) Let
$$T_1g(z)=\frac{\mathrm{\partial}^2g(z)}{\mathrm{\partial}z^2_1}-z_1g(z),\;T_2g(z)=\frac{\mathrm{\partial}^2g(z)}{\mathrm{\partial}z^2_2}-z_2g(z),\;g\in H(\mathbb{C}^2).$$
Then $\bigcap_{j=1}^2\mathrm{ker}\,T_j$ contain, for example, the function $\mathrm{Ai}(z_1)\cdot\mathrm{Ai}(z_2)$, where $\mathrm{Ai}$ is Airy function. Moreover, (\ref{eq200}) holds. Hence, $T_1$ and $T_2$ are frequently hypercyclic by Theorem \ref{th1}.

\bigskip
\textbf{Acknowledgement.} The author is grateful to the referee for helpful comments and suggestions.

\end{document}